\documentclass[12pt, a4paper, leqno]{amsart}
\usepackage{amsmath, amsthm, amssymb}
\usepackage[top=1truein, bottom=1truein, hmargin={1truein, 1truein}]{geometry}
\usepackage{indentfirst}
\usepackage{bm}
 \newtheorem{theorem}{Theorem}[section]
 \newtheorem{cor}[theorem]{Corollary}
 
 \newtheorem{prop}[theorem]{Proposition}

 \theoremstyle{definition}
 
 \theoremstyle{remark}
 \newtheorem{remark}[theorem]{Remark}
 
 \numberwithin{equation}{section}

\begin{document}

%
%
%
%
%
%
%
%
%

\title[Duality theorems of multiple zeta values with parameters]
 {Duality theorems of multiple zeta values \\  with parameters}

\author{Chan-Liang Chung}

\address{Institute of Mathematics, Academia Sinica, 6F, Astronomy-Mathematics Building, No. 1, Sec. 4, Roosevelt Road, Taipei 10617, Taiwan(R.O.C.) }

\email{andrechung@gate.sinica.edu.tw}


\author{Minking Eie}
\address{Department of Mathematics, National~Chung~Cheng University, 168 University Rd. Minhsiung, Chiayi 62145, Taiwan(R.O.C.) }
\email{minking@math.ccu.edu.tw}

\subjclass{Primary: 40B05; Secondary: 11M40,11M06,33E20.}

\keywords{multiple zeta value, duality theorem, sum formula.}

\date{Sep. 1, 2017}

\begin{abstract}
In this paper, we introduce the method of adding additional factors and a parameter to multiple zeta values and prove some generalizations of the duality theorem and several relations among multiple zeta values. In particular, we are able to evaluate some special (truncated) sums in terms of Riemann zeta values of different weights.
\end{abstract}

\maketitle

\section{Introduction}
We fix a positive integer $r$ and put
\[
T_{\bm{k}}(r)=\{\bm{k}=(k_1, k_2, \ldots, k_r)\in \mathbb{N}^r \;|\; 1\leq k_1<k_2<\cdots <k_r\}.
\]
For an $r$-tuple of positive integers $\bm{\alpha}=(\alpha_1, \alpha_2, \ldots, \alpha_r)$ with $\alpha_r\geq 2$, a multiple zeta value (or $r$-fold Euler sum) (cf. \cite{BBBL01}) of depth $r$ and weight $|\bm{\alpha}|=\alpha_1+\alpha_2+\cdots+\alpha_r$ is defined as
\[
\zeta(\alpha_1, \alpha_2, \ldots, \alpha_r)=\sum_{T_{\bm{k}}(r)}k_1^{-\alpha_1}k_2^{-\alpha_2}\cdots k_r^{-\alpha_r}.
\]
For our convenience, we let $\{1\}^k$ be $k$ repetitions of $1$. Due to Kontsevich \cite{D91}, multiple zeta values can be expressed as iterated integrals (or Drinfel'd integrals) over the simplex of dimension $|\bm{\alpha}|$ defined by $E_{|\bm{\alpha}|}:\;0<t_1<t_2<\cdots<t_{|\bm{\alpha}|}<1$. Indeed, we have
\[
\zeta(\alpha_1, \alpha_2, \ldots, \alpha_r)=\int_{E_{|\bm{\alpha}|}}\Omega_1\Omega_2 \cdots \Omega_{|\bm{\alpha}|}
\]
with
\[
   \Omega_{j}=\left\{\begin{array}{ll} dt_j/(1-t_j) & \mbox{if}\; j=1, \alpha_1+1, \alpha_1+\alpha_2+1, \ldots, \alpha_1+\alpha_2+\ldots +\alpha_{r-1}+1; \\
   dt_j/t_j &\mbox{otherwise}. \end{array} \right.
\]

Some particular multiple zeta values and sums of multiple zeta values can be further expressed as double integrals.
\begin{prop}\cite[page 120]{E09} \label{prop1.1}
For a pair of nonnegative integers $m$ and $n$, we have
\begin{equation*}
\begin{split}
\zeta(\{1\}^m, n+2)=&\frac{1}{m!n!}\int_{E_2}\left(\log\frac{1}{1-t_1}\right)^m \left(\log\frac{1}{t_2}\right)^n \frac{dt_1dt_2}{(1-t_1)t_2}\\
=&\frac{1}{m!n!}\int_{E_2}\left(\log\frac{1-t_1}{1-t_2}\right)^m \left(\log\frac{1}{t_2}\right)^n \frac{dt_1dt_2}{(1-t_1)t_2}.
\end{split}
\end{equation*}
\end{prop}
\begin{prop}\cite[page 120]{E09}\label{prop1.2}
For nonnegative integers $p, q, r$ and $\ell$, we have
\begin{equation*}
\begin{split}
&\sum_{|\bm{\alpha}|=q+r+1}\zeta(\{1\}^p, \alpha_1, \ldots, \alpha_r, \alpha_{r+1}+\ell+1)\\
=&\frac{1}{p!q!r!\ell!}\int_{E_2}\left(\log\frac{1}{1-t_1}\right)^p \left(\log\frac{1-t_1}{1-t_2}\right)^r \left(\log\frac{t_2}{t_1}\right)^q \left(\log\frac{1}{t_2}\right)^{\ell} \frac{dt_1dt_2}{(1-t_1)t_2}.
\end{split}
\end{equation*}
\end{prop}

For a pair of positive integers $p$ and $q$, the multiple zeta value $\zeta(\{1\}^{p-1}, q+1)$ can be expressed as iterated integral
\[
\int_{E_{p+q}}\prod_{j=1}^p \frac{dt_j}{1-t_j} \prod_{i=p+1}^{p+q} \frac{dt_i}{t_i}.
\]
The change of variables $u_1=1-t_{p+q}, \; u_2=1-t_{p+q-1}, \ldots, u_{p+q}=1-t_1$ then transforms the iterated integral into
\[
\int_{E_{p+q}}\prod_{j=1}^{q}\frac{du_j}{1-u_j}\prod_{i=q+1}^{p+q}\frac{du_i}{u_i}
\]
which is $\zeta(\{1\}^{q-1}, p+1)$. So that we have $\zeta(\{1\}^{p-1}, q+1)=\zeta(\{1\}^{q-1}, p+1)$.
This is called the Drinfel'd duality theorem and it can be extended to a vectorized version (cf.  \cite{EL16}). For $n$ pairs of positive integers $p_1, q_1;\; p_2, q_2; \ldots;\; p_n, q_n$, let 
\[
\bm{k}=(\{1\}^{p_1-1}, q_1+1, \{1\}^{p_2-1}, q_2+1, \ldots, \{1\}^{p_n-1}, q_n+1)
\]
and let $\bm{k}'$ be its dual defined by 
\[
\bm{k}'=(\{1\}^{q_n-1}, p_n+1, \{1\}^{q_{n-1}-1}, p_{n-1}+1, \ldots, \{1\}^{q_1-1}, p_1+1).
\]
Then the general duality theorem asserts that $\zeta(\bm{k})=\zeta(\bm{k}')$ \cite{Z94}. Ohno's generalization of the duality and sum formulas \cite{O99} asserts further: 
\begin{equation}\label{eq1.1}
\sum_{|\bm{c}|=m}\zeta(\bm{k}+\bm{c})=\sum_{|\bm{d}|=m}\zeta(\bm{k}'+\bm{d}),
\end{equation}
for any nonnegative integer $m$. In particular, for a pair of positive integers $p$ and $q$, we have
\begin{equation}\label{eq1.2}
\sum_{|\bm{\alpha}|=p+m}\zeta(\alpha_1,\ldots, \alpha_{p-1}, \alpha_p+q)=\sum_{|\bm{\beta}|=q+m}\zeta(\beta_1,\ldots, \beta_{q-1}, \beta_q+p).
\end{equation}
Note that the special case $q=1$ of (\ref{eq1.2}) then gives the sum formula of multiple zeta values.
\begin{prop}[The Sum Formula \cite{G97}]\label{prop1.3}
For integers $m>p\geq 0$, we have
\[
\sum_{|\bm{\alpha}|=m}\zeta(\alpha_1, \ldots, \alpha_{p-1}, \alpha_p+1)=\zeta(m+1).
\]
\end{prop}

In this paper, we add further factors to multiple zeta values with parameters and produce the following theorems.
\begin{theorem}\label{thm1}
Suppose that $p$ and $q$ are positive integers and $r$ is a nonnegative integer. Then for any real number $a>-1$ and nonnegative integer $m$, we have
\begin{equation}\label{eq1.3}
\begin{split}
&\sum_{\substack{|\bm{\alpha}|=p+m \\ T_{\bm{k}}(p)}}(k_1+a)^{-\alpha_1}(k_2+a)^{-\alpha_2}\cdots (k_p+a)^{-\alpha_p}(k_p+r)^{-q}\\
=&\frac{1}{r!}\sum_{\substack{|\bm{\beta}|=q+m \\ T_{\bm{\ell}}(q)}}\frac{\ell_1(\ell_1+1)\cdots(\ell_1+r-1)}{(\ell_1+a)^{\beta_1}(\ell_2+a)^{\beta_2}\cdots(\ell_q+a)^{\beta_q}}\sum_{j=0}^r (-1)^j \binom{r}{j}(\ell_q+j)^{-p}.
\end{split}
\end{equation}
\end{theorem}

In particular, taking $q=1$ and $a=r$ in (\ref{eq1.3}) gives the following corollary.
\begin{cor}\label{cor1.5}
For any positive integer $p$ and nonnegative integers $m$ and $r$ with  $m+p\geq r+1$, the truncated sum
\begin{equation*}
\sum_{\substack{|\bm{\alpha}|=p+m \\ T_{\bm{k}}(p)}}(k_1+r)^{-\alpha_1}\cdots(k_{p-1}+r)^{-\alpha_{p-1}}(k_p+r)^{-\alpha_p-1}
\end{equation*}
is equal to
\begin{equation*}
\frac{1}{r!}\sum_{\ell=1}^{\infty} \frac{\ell(\ell+1)\cdots(\ell+r-1)}{(\ell+r)^{m+1}}\sum_{j=0}^r (-1)^j \binom{r}{j}(\ell+j)^{-p}.
\end{equation*}
\end{cor}

The outline of this paper is as follows. In Section 2, we shall introduce the method of adding some factors to multiple zeta values with parameter and give a proof of Theorem \ref{thm1} according to this way. In Section 3, based on known duality relations, we can produce new identities concerning duality. We obtain a simple relation by starting with a special truncated sum in Section 4 and give a brief summary of this paper in the final section.
 
\section{Multiple zeta values with parameter and additional factors}
For different purposes, different factors such as
\[
\left(\frac{t_1}{t_{|\bm{\alpha}|}}\right)^a, \left(\frac{1-t_{|\bm{\alpha}|}}{1-t_1}\right)^b, \left(\frac{1}{1-t_1}\right)^c, t_{|\bm{\alpha}|}^d
\]
are attached to iterated integral representations of multiple zeta values to produce new multiple zeta values when differentiating with respect to the parameters. For example, the iterated integral 
\[
\int_{E_{r+1}}\left(\frac{t_1}{t_{r+1}}\right)^a \left(\prod_{j=1}^r \frac{dt_j}{1-t_j}\right) \frac{dt_{r+1}}{t_{r+1}}
\]
obtained from the iterated integral of $\zeta(\{1\}^{r-1}, 2)$ with the additional factor $(t_1/t_{r+1})^a$, can be evaluated as
\[
\sum_{T_{\bm{k}}(r)}\frac{1}{(k_1+a)(k_2+a)\cdots (k_r+a)k_r}.
\]
In a similar manner we have for positive integers $p$ and $q$
\begin{equation}\label{eq2.1}
\begin{split}
\int_{E_{p+q}}\left(\frac{t_1}{t_{p+1}}\right)^a \prod_{j=1}^p \frac{dt_j}{1-t_j} \prod_{i=p+1}^{p+q}\frac{dt_i}{t_i}=\sum_{T_{\bm{k}}(r)}[(k_1+a)\cdots(k_p+a)]^{-1}k_p^{-q}.
\end{split}
\end{equation}

On the other hand, we add extra factors such as $t_j^{r_j}, (1-t_j)^{r_j}$ with $r_j$ integers to iterated integrals of multiple zeta values to produce sums of multiple zeta values of different weights. For examples,
\[
\int_{E_2}t_2^2\frac{dt_1dt_2}{(1-t_1)t_2}=\sum_{k=1}^{\infty}\frac{1}{k(k+2)}=\frac{3}{4}
\]
and
\[
\int_{E_3}\frac{dt_1}{(1-t_1)^2}\frac{dt_2}{t_2}\frac{dt_3}{t_3}=\sum_{k=1}^{\infty}\frac{k}{k^3}=\zeta(2).
\]

Here we consider a special multiple zeta value with a parameter $a$ and an additional factor related to the multiple zeta value $\zeta(\{1\}^{p-1}, q+1)$, namely
\begin{equation}\label{eq2.2}
\begin{split}
I(p,q;a,r)=\int_{E_{p+q}}(t_{p+1})^r \left(\frac{t_1}{t_{p+1}}\right)^a \prod_{j=1}^p \frac{dt_j}{1-t_j} \prod_{i=p+1}^{p+q} \frac{dt_i}{t_i}.
\end{split}
\end{equation}
Note that $I(p,q;a,r)$ can be evaluated easily as
\begin{equation*}
\begin{split}
\sum_{T_{\bm{k}}(p)} [(k_1+a)(k_2+a)\cdots (k_p+a)]^{-1}(k_p+r)^{-q}.
\end{split}
\end{equation*}

Now we are ready to prove our main theorem.
\begin{proof}[Proof of Theorem \ref{thm1}]
Let $I(p,q;a,r)$ be defined as (\ref{eq2.2}). Fix $t_1, t_{p+1}$ and integrate with respect to the remaining variables $t_2, \ldots, t_p$ and $t_{p+2}, \ldots, t_{p+q}$ to express the iterated integral as the double integral
\[
\frac{1}{(p-1)!(q-1)!}\int_{E_2}t_2^r \left(\frac{t_1}{t_2}\right)^a \left(\log\frac{1-t_1}{1-t_2}\right)^{p-1} \left(\log\frac{1}{t_2}\right)^{q-1} \frac{dt_1dt_2}{(1-t_1)t_2}.
\]
Under the change of variables:
\[
x_1=\log\frac{1-t_1}{1-t_2}, \;x_2=\log\frac{1}{t_2},
\]
the double integral is transformed into
\[
\frac{1}{(p-1)!(q-1)!}\int_{D_2}x_1^{p-1}x_2^{q-1}\left(e^{x_1}+e^{x_2}-e^{x_1+x_2}\right)^a e^{-rx_2}dx_1dx_2,
\]
where
\[
D_2=\{(x_1, x_2) \;|\; x_1>0, \;x_2>0,\; e^{x_1}+e^{x_2}-e^{x_1+x_2}>0\}.
\]
Now interchange $x_1$ and $x_2$ to yield the double integral with the same value
\[
\frac{1}{(p-1)!(q-1)!}\int_{D_2}x_1^{q-1}x_2^{p-1}\left(e^{x_1}+e^{x_2}-e^{x_1+x_2}\right)^a e^{-rx_1}dx_1dx_2.
\]
With the same change of variables
\[
x_1=\log\frac{1-u_1}{1-u_2}, \; x_2=\log\frac{1}{u_2},
\]
we conclude that
\begin{equation}\label{eq2.3}
\begin{split}
I(p,q;a,r)&=\frac{1}{(p-1)!(q-1)!}\int_{E_2} \left(\frac{1-u_2}{1-u_1}\right)^r \left(\frac{u_1}{u_2}\right)^a\left(\log\frac{1-u_1}{1-u_2}\right)^{q-1}\\
&\hspace{32mm}\times \left(\log\frac{1}{u_2}\right)^{p-1} \frac{du_1du_2}{(1-u_1)u_2}\\
&=\int_{E_{p+q}}\left(\frac{1-u_{q+1}}{1-u_1}\right)^r\left(\frac{u_1}{u_{q+1}}\right)^a \left(\prod_{j=1}^q\frac{du_j}{1-u_j}\right) \left(\prod_{k=q+1}^{p+q}\frac{du_k}{u_k}\right).
\end{split}
\end{equation}

In light of the expression
\[
\frac{1}{(1-u_1)^{r+1}}=\sum_{\ell_1=1}^{\infty} \binom{\ell_1+r-1}{r}u_1^{\ell_1-1},
\]
the double integral in the right hand side of (\ref{eq2.3}) can be evaluated as
\begin{equation*}
\begin{split}
\frac{1}{r!}\sum_{T_{\bm{\ell}}(q)}\frac{\ell_1(\ell_1+1)\cdots (\ell_1+r-1)}{(\ell_1+a)(\ell_2+a)\cdots (\ell_q+a)}\sum_{j=0}^r (-1)^j \binom{r}{j} (\ell_q+j)^{-p}.
\end{split}
\end{equation*}
This leads to the identity
\begin{equation*}
\begin{split}
&\sum_{T_{\bm{k}}(p)}[(k_1+a)(k_2+a)\cdots (k_p+a)]^{-1}(k_p+r)^{-q}\\
=&\frac{1}{r!}\sum_{T_{\bm{\ell}}(q)}\frac{\ell_1(\ell_1+1)\cdots (\ell_1+r-1)}{(\ell_1+a)(\ell_2+a)\cdots (\ell_q+a)}\sum_{j=0}^r (-1)^j \binom{r}{j} (\ell_q+j)^{-p}.
\end{split}
\end{equation*}
Differentiating both sides of the above identity with respect to $a$ for $m$ times, we obtain our assertion.
\end{proof}

Taking $r=0$ into Theorem \ref{thm1} and then letting $a=0$, we get the special case of the duality and the formula (\ref{eq1.2}). With the similar consideration of multiple zeta values with parameters, Ohno's theorem (\ref{eq1.1}) can be obtained from the differentiations of the more general identity 
\begin{equation}\label{eq2.4}
\begin{split}
&\sum_{T_{\bm{k}}(|\bm{p}|)} \frac{1}{(k_1+a)(k_2+a)\cdots (k_{|\bm{p}|}+a)k_{p_1}^{q_1}k_{p_1+p_2}^{q_2}\cdots k_{|\bm{p}|}^{q_n}}\\
=&\sum_{T_{\bm{\ell}}(|\bm{q}|)}\frac{1}{(\ell_1+a)(\ell_2+a)\cdots (\ell_{|\bm{q}|}+a)\ell_{q_n}^{p_n}\ell_{q_{n-1}+q_n}^{p_{n-1}}\cdots \ell_{|\bm{q}|}^{p_1}} ,
\end{split}
\end{equation}
where the real number $a>-1$ and $\bm{p}=(p_1, p_2, \ldots, p_n)$, $\bm{q}=(q_1, q_2, \ldots, q_n)$ are $n$-tuples of positive integers. In fact, the left hand side of (\ref{eq2.4}) is equal to the vectorized form of (\ref{eq2.1}), \textit{i.e.} the integral obtained from the iterated integral of $\zeta(\{1\}^{p_1-1}, q_1+1, \ldots, \{1\}^{p_n-1}, q_n+1)$ with an additional factor (cf. \cite[page 76-82]{E13})
\begin{equation*}
\begin{split}
&\int_{E_{|\bm{p}|+|\bm{q}|}}\left(\frac{t_1t_{s_1+1}\cdots t_{s_{n-1}+1}}{t_{p_1+1}t_{s_1+p_2+1}\cdots t_{s_{n-1}+p_n+1}}\right)^a\prod_{j_1=1}^{p_1}\frac{dt_{j_1}}{1-t_{j_1}}\prod_{i_1=p_1+1}^{s_1}\frac{dt_{i_1}}{t_{i_1}}\\
&\times \prod_{j_2=s_1+1}^{s_1+p_2}\frac{dt_{j_2}}{1-t_{j_2}}\prod_{i_2=s_1+p_2+1}^{s_2}\frac{dt_{i_2}}{t_{i_2}}\;\cdots\prod_{j_n=s_{n-1}+1}^{s_{n-1}+p_n}\frac{dt_{j_n}}{1-t_{j_n}}\prod_{i_n=s_{n-1}+p_n+1}^{s_n}\frac{dt_{i_n}}{t_{i_n}},
\end{split}
\end{equation*}
where $s_m=p_1+q_1+\cdots+p_m+q_m$. So the identity (\ref{eq1.3}) can be viewed as a generalization of the special case of Ohno's theorem. 

\section{Duality among restricted sums}
For nonnegative integers $p, q, r$, the restricted sum formula \cite{ELO09} asserted that
\[
\sum_{|\bm{\alpha}|=q+r+1}\zeta(\{1\}^p, \alpha_1, \ldots, \alpha_{r}, \alpha_{r+1}+1)=\sum_{|\bm{\beta}|=p+r+1}\zeta(\beta_0, \beta_1, \ldots, \beta_p+q+1).
\]
There is another duality among restricted sums \cite{CEL16, EL16}
\begin{equation*}
\begin{split}
&\sum_{|\bm{\alpha}|=q+r+1}\zeta(\{1\}^p, \alpha_1, \ldots, \alpha_r, \alpha_{r+1}+1)\\
=&\sum_{|\bm{\beta}|=p+r+1}\zeta(\{1\}^q, \beta_1, \ldots, \beta_r, \beta_{r+1}+1).
\end{split}
\end{equation*}
By Proposition \ref{prop1.2}, the above relation is given by
\begin{equation*}
\begin{split}
&\frac{1}{p!q!r!}\int_{E_2}\left(\log\frac{1}{1-t_1}\right)^p \left(\log\frac{1-t_1}{1-t_2}\right)^r \left(\log\frac{t_2}{t_1}\right)^q \frac{dt_1dt_2}{(1-t_1)t_2}\\
=&\frac{1}{p!r!q!}\int_{E_2}\left(\log\frac{1-u_1}{1-u_2}\right)^p \left(\log\frac{u_2}{u_1}\right)^r \left(\log\frac{1}{u_2}\right)^q \frac{du_1du_2}{(1-u_1)u_2}\\
=&\frac{1}{q!r!p!}\int_{E_2}\left(\log\frac{1}{1-v_1}\right)^q \left(\log\frac{1-v_1}{1-v_2}\right)^r \left(\log\frac{v_2}{v_1}\right)^p \frac{dv_1dv_2}{(1-v_1)v_2}.
\end{split}
\end{equation*}
We now prove the following extension.
\begin{theorem}\label{thm3}
Suppose that $p,q,m$ and $r$ are nonnegative integers. Then the following series are equal.

\begin{equation}\label{eq3.1}
\begin{split}
\sum_{\substack{|\bm{\alpha}|=q+r+1 \\ T_{\bm{k}}(p+r+1)}}\frac{(k_1k_2\cdots k_p k_{p+r+1})^{-1}}{(k_{p+1}+m)^{\alpha_1}(k_{p+2}+m)^{\alpha_2}\cdots (k_{p+r+1}+m)^{\alpha_{r+1}}};
\end{split}
\end{equation}

\begin{equation}\label{eq3.2}
\begin{split}
\sum_{\substack{|\bm{\beta}|=p+r+1 \\ T_{\bm{\ell}}(p+1)}} \ell_1^{-\beta_1}\ell_2^{-\beta_2}\cdots \ell_{p+1}^{-\beta_{p+1}}(\ell_{p+1}+m)^{-q-1};
\end{split}
\end{equation}

\begin{equation}\label{eq3.3}
\begin{split}
&\sum_{j=0}^m (-1)^j \binom{m}{j} \sum_{\substack{|\bm{\beta}|=p+r+1 \\ T_{\bm{\ell}}(q+r+1)
}}(\ell_1\ell_2\cdots \ell_q)^{-1}(\ell_{q+1}+j)^{-\beta_1}(\ell_{q+2}+j)^{-\beta_2}\\
&\hspace{30mm}\times\cdots \times(\ell_{q+r}+j)^{-\beta_{q}} (\ell_{q+r+1}+j)^{-\beta_{q+1}-1}.
\end{split}
\end{equation}
\end{theorem}
\begin{proof}
We express (\ref{eq3.1}) as the double integral
\[
\frac{1}{p!q!r!}\int_{E_2}\left(\frac{t_1}{t_2}\right)^m\left(\log\frac{1}{1-t_1}\right)^p \left(\log\frac{1-t_1}{1-t_2}\right)^{r} \left(\log\frac{t_2}{t_1}\right)^q \frac{dt_1dt_2}{(1-t_1)t_2}.
\]
Under the change of variables:
\[
x_1=\log\frac{1}{1-t_1},\hspace{5mm} x_2=\log\frac{t_2}{t_1},
\]
the double integral is transformed into
\begin{equation}\label{int1}
\frac{1}{p!q!r!}\int_{D_2}x_1^p x_2^q \left(\log\frac{1}{e^{x_1}+e^{x_2}-e^{x_1+x_2}}\right)^{r} e^{-mx_2}dx_1dx_2
\end{equation}
with $D_2=\{(x_1, x_2) \;|\; x_1>0, \;x_2>0,\; e^{x_1}+e^{x_2}-e^{x_1+x_2}>0\}$. Now with another change of variables:
\[
x_1=\log\frac{1-u_1}{1-u_2},\hspace{5mm} x_2=\log\frac{1}{u_2},
\]
the double integral is transformed back to
\[
\frac{1}{p!q!r!}\int_{E_2}u_2^m\left(\log\frac{1-u_1}{1-u_2}\right)^p \left(\log\frac{u_2}{u_1}\right)^{r} \left(\log\frac{1}{u_2}\right)^q \frac{du_1du_2}{(1-u_1)u_2}.
\]
In terms of multiple zeta values, it is exactly the integral representation of (\ref{eq3.2}). This proves the equality of $(\ref{eq3.1})$ and $(\ref{eq3.2})$.

On the other hand, we interchange $x_1$ and $x_2$ in (\ref{int1}) to yield the double integral of the same value
\[
\frac{1}{p!q!r!}\int_{D_2}x_1^p x_2^q \left(\log\frac{1}{e^{x_1}+e^{x_2}-e^{x_1+x_2}}\right)^r e^{-mx_2}dx_1dx_2.
\]
This corresponds to the following double integral 
\[
\frac{1}{q!r!p!}\int_{E_2}(1-v_1)^m\left(\log\frac{1}{1-v_1}\right)^q \left(\log\frac{1-v_1}{1-v_2}\right)^r \left(\log\frac{v_2}{v_1}\right)^p \frac{dv_1dv_2}{(1-v_1)v_2}
\]
after making the change of variables:
\[
x_1=\log\frac{1}{1-v_1}, x_2=\log\frac{v_2}{v_1}.
\]
In terms of multiple zeta values, it is equal to (\ref{eq3.3}).
\end{proof}
\begin{remark}
In summary, in the proof of previous theorem, we obtain that
\begin{equation}
\begin{split}
&\frac{1}{p!q!r!}\int_{E_2}\left(\frac{t_1}{t_2}\right)^m\left(\log\frac{1}{1-t_1}\right)^p \left(\log\frac{1-t_1}{1-t_2}\right)^r \left(\log\frac{t_2}{t_1}\right)^q \frac{dt_1dt_2}{(1-t_1)t_2}\\
=&\frac{1}{p!r!q!}\int_{E_2}u_2^m\left(\log\frac{1-u_1}{1-u_2}\right)^p \left(\log\frac{u_2}{u_1}\right)^r \left(\log\frac{1}{u_2}\right)^q \frac{du_1du_2}{(1-u_1)u_2}\\
=&\frac{1}{q!r!p!}\int_{E_2}(1-v_1)^m\left(\log\frac{1}{1-v_1}\right)^q \left(\log\frac{1-v_1}{1-v_2}\right)^r \left(\log\frac{v_2}{v_1}\right)^p \frac{dv_1dv_2}{(1-v_1)v_2}.
\end{split}
\end{equation}

The parameter $m$ here can be a real number so that we can perform differentiations with respect to $m$ to produce more general identities concerning duality.
\end{remark}

\section{A truncated sum}
The special truncated sum (see Corollary \ref{cor1.5})
\[
\sum_{\substack{|\bm{\alpha}|=m+p \\ T_{\bm{k}}(p)}}(k_1+1)^{-\alpha_1}\cdots (k_{p-1}+1)^{-\alpha_{p-1}}(k_p+1)^{-\alpha_p-1}
\]
has the value
\[
\sum_{\ell=1}^{\infty} \frac{\ell}{(\ell+1)^{m+1}}\left[\frac{1}{\ell^p}-\frac{1}{(\ell+1)^p}\right]
\]
or
\[
\sum_{\ell=1}^{\infty}\frac{1}{\ell^{p-1}(\ell+1)^{m+1}}-\zeta(m+p)+\zeta(m+p+1).
\]
This implies, in particular, the partial sum corresponding to $k_1=1$
\[
S=\sum_{\substack{|\bm{\alpha}|=m+p \\ 1<k_2<\cdots<k_p}}1^{-\alpha_1}k_2^{-\alpha_2}\cdots k_{p-1}^{-\alpha_{p-1}}k_p^{-\alpha_p-1}
\]
has the value
\[
\zeta(m+p)-\sum_{\ell=1}^{\infty}\frac{1}{\ell^{p-1}(\ell+1)^{m+1}}
\]
when $m+p\geq 2$. Of course, the value is a linear combination of single zeta values along with constants.

For $1\leq j<p$, we let
\[
S_j=\sum_{\substack{|\bm{\alpha}|=m+p \\ 1\leq k_{j+1}<\cdots<k_p}}1^{-\alpha_1-\alpha_2-\cdots-\alpha_j}k_{j+1}^{-\alpha_{j+1}}\cdots k_{p-1}^{-\alpha_{p-1}}k_{p}^{-\alpha_p-1}
\]
and $S_p=\binom{m+p-1}{m}$. Then $S=S_1-S_2+S_3-\cdots +(-1)^{p-1}S_p$.
\begin{theorem}
For a pair of positive integers $m$ and $p$, then we have
\[
S_1-S_2+S_3-\cdots+(-1)^{p-1}S_p=\zeta(m+p)-\sum_{\ell=1}^{\infty}\frac{1}{\ell^{p-1}(\ell+1)^{m+1}}.
\]
\end{theorem}
\section{A final remark}
When multiple zeta values are expressed as iterated integrals, most of duality theorems of multiple zeta values are obtained from the change of variables: $\bm{t}\rightarrow 1-\bm{t}=\bm{u}$ or componentwise
\[
u_1=1-t_{|\bm{\alpha}|}, u_2=1-t_{|\bm{\alpha}|-1}, \ldots, u_{|\bm{\alpha}|}=1-t_1.
\]
This also includes multiple zeta values or sums of multiple zeta values expressed as double integrals. For example, we have
\begin{equation*}
\begin{split}
&\frac{1}{m!n!}\int_{E_2}\left(\log\frac{1-t_1}{1-t_2}\right)^m \left(\log\frac{1}{t_2}\right)^n \frac{dt_1dt_2}{(1-t_1)t_2}\\
=&\frac{1}{m!n!}\int_{E_2}\left(\log\frac{1}{1-u_1}\right)^n \left(\log\frac{u_2}{u_1}\right)^m \frac{du_1du_2}{(1-u_1)u_2}.
\end{split}
\end{equation*}
On the other hand, when multiple zeta values or sums of multiple zeta values are expressed in double integrals, the change of variables
\[
x_1=\log\frac{1}{1-t_1}, x_2=\log\frac{t_2}{t_1}
\]
or
\[
x_1=\log\frac{1-t_1}{1-t_2}, x_2=\log\frac{1}{t_2}
\]
provides an alternative way to obtain the duals. Especially, when it is hard to evaluate the duals simply obtained from the change of variables $\bm{t}\rightarrow 1-\bm{t}$.

\subsection*{Acknowledgment}
The authors thank the anonymous referee for careful reading the manuscript and for very helpful comments and suggestions.

\providecommand{\bysame}{\leavevmode \hbox to3em%
{\hrulefill}\thinspace}

\end{document}